\documentclass[12pt]{amsart}
\usepackage{amsfonts}
\usepackage{amsmath}
\usepackage{amssymb}
\usepackage[margin=1.2in]{geometry}
\usepackage{bbm}
\usepackage{tikz}
\allowdisplaybreaks[4]

\newtheorem{theorem}{Theorem}[section]

\newtheorem{lemma}[theorem]{Lemma}

\theoremstyle{definition}

\newtheorem{remark}[theorem]{Remark}

\numberwithin{equation}{section}
\numberwithin{theorem}{section}

\renewcommand{\Re}{\operatorname{Re}}
\renewcommand{\Im}{\operatorname{Im}}

\numberwithin{equation}{section}

\makeatletter
\@namedef{subjclassname@2020}{
  \textup{2020} Mathematics Subject Classification}
\makeatother

\begin{document}
\title[Density hypothesis for $L$-functions associated with cusp forms]{On the density hypothesis for $L$-functions associated with holomorphic cusp forms}

\author[B. Chen]{Bin Chen}
\thanks{B. Chen gratefully acknowledges support by the China Scholarship Council (CSC)}
\address{B. Chen\\ Department of Mathematics: Analysis, Logic and Discrete Mathematics\\ Ghent University\\ Krijgslaan 281\\ B 9000 Ghent\\ Belgium}
\email{bin.chen@UGent.be}
\author[G. Debruyne]{Gregory Debruyne}
\thanks{G. Debruyne gratefully acknowledges support by Ghent University through a postdoctoral fellowship (grant number 01P13022)}
\address{G. Debruyne\\ Department of Mathematics: Analysis, Logic and Discrete Mathematics\\ Ghent University\\ Krijgslaan 281\\ B 9000 Ghent\\ Belgium}
\email{gregory.debruyne@UGent.be}
\author[J. Vindas]{Jasson Vindas}
\thanks{The work of J. Vindas was supported by the Research Foundation--Flanders, through the FWO-grant number G067621N}
\address{J. Vindas\\ Department of Mathematics: Analysis, Logic and Discrete Mathematics\\ Ghent University\\ Krijgslaan 281\\ B 9000 Ghent\\ Belgium}
\email{jasson.vindas@UGent.be}
\subjclass[2020]{Primary 11M26,  11M41; Secondary 11M06, 11F11, 11F66.}
\keywords{Density hypothesis; Zero-density estimates; Riemann zeta function; Holomorphic modular forms}
\begin{abstract} 
We study the range of validity of the density hypothesis for the zeros of $L$-functions associated with cusp Hecke eigenforms $f$ of even integral weight and prove that $N_{f}(\sigma, T) \ll T^{2(1-\sigma)+\varepsilon}$ holds for $\sigma \geq 1407/1601$. This improves upon a result of Ivi\'{c}, who had previously shown the zero-density estimate in the narrower range  $\sigma\geq 53/60$.  Our result relies on an improvement of the large value estimates for Dirichlet polynomials based on mixed moment estimates for the Riemann zeta function. The main ingredients in our proof are the Hal\'{a}sz-Montgomery inequality, Ivi\'{c}'s mixed moment bounds for the zeta function, Huxley's subdivision argument, Bourgain's dichotomy approach, and Heath-Brown's bound for double zeta sums. 

\end{abstract}

\maketitle

\section{Introduction}\label{intro}

Zero-density estimates for the Riemann zeta function and 
 $L$-functions play a central role in analytic number theory. They have important arithmetic consequences; see for instance \cite[Chap. 12]{Ivicbook} and \cite[Chap. 15]{MontgomeryBook} for an overview of applications in prime number theory.

 Let
  $N(\sigma, T) $ denote the number of zeros $\rho = \beta + it$ of the Riemann zeta function $\zeta(s)$ in the rectangle $\sigma \leq \beta \leq 1$, $|t| \leq T$. In 1937, Ingham \cite{Ingham1937} connected estimates of the form ($ c,D > 0$)
\begin{equation} \label{defUB}
	N(\sigma , T) \ll T^{c(1-\sigma)}\log^{D} T, \qquad \mbox{uniformly for } \frac{1}{2}\leq \sigma \leq1,
\end{equation}
with the behavior of primes in short intervals. In fact, one can prove that \eqref{defUB} implies the PNT in the form $\psi(x+h)-\psi(x)= h (1+o(1))$ for $h\gg x^{1+\varepsilon-1/c}$ as $x\to\infty$. Note that the estimate \eqref{defUB} with $c=2$ essentially provides the same result as the Riemann hypothesis. 
As this turns out to be the case for many other arithmetic results that are also obtainable from the Lindel\"{o}f or the Riemann hypothesis, an inequality of the sort

\begin{equation}\label{DefDH2}
	N(\sigma , T) \ll T^{2(1-\sigma)+\varepsilon},
\end{equation}
has become known as the \emph{density hypothesis}. While a proof that the density hypothesis holds uniformly for $1/2\leq \sigma \leq1$ seems to be out of reach by present methods, there has been substantial progress towards maximizing 
its range of validity. Montgomery showed \cite{Montgomery1969} that the density hypothesis \eqref{DefDH2} holds in the range $\sigma \geq 9/10$. The range of validity was subsequently improved (cf. \cite{Huxley1972,Ramachandra1975,F-V1973,Huxley1975,Jutila1977}), and the current record is due to Bourgain \cite{Bourgain2000}, who showed that \eqref{DefDH2} is valid for $\sigma\geq 25/32=0.78125.$

 It is also natural to consider zero-density estimates for Dirichlet $L$-functions \cite{B-R1974, Heath-Brown1979, Heath-Brown1979sec,  Huxley1973, Huxley and Jutila1977, Jutila1972, Jutila1977} and  for $L$-functions associated with modular forms \cite{Hough2012,Ivic1992,Kamiya1998,Ren-Zhang2019, Y.Y}. We are interested in studying the density hypothesis for the latter case. So, let $f(z)=\sum_{n=1}^{\infty}a_{f}(n)e^{2\pi i z}$ be a holomorphic cusp form of even integral weight $\kappa$ for the full modular group $\mbox{SL}(2,\mathbb{Z})$. We assume that $f$ is a Hecke eigenform \cite{Apostolbook,DSbook}, and that it is normalized, i.e. $a_{f}(1)=1$. We set $\lambda_{f}(n) = a_f (n)n^{-\frac{\kappa-1}{2}}$ and notice that this multiplicative function satisfies $|\lambda_{f}(n)|\leq d(n)$, where $d(n)$ is the divisor function, an inequality that was shown by Deligne \cite[Thm. 8.2, p. 302]{P. Deligne} as a consequence of his proof of Weil's conjectures. The $L$-function $L(s,f)$ associated with $f$ is defined as
\[
L(s,f)= \sum_{n=1}^{\infty}\frac{\lambda_{f}(n)}{n^{s}}=\prod_{p} \left( 1- \lambda_{f}(p)p^{-s}+p^{-2s}\right)^{-1}, \ \ \ \Re(s)>1.
\]
A classical result of Hecke establishes that $L(s,f)$ extends to the whole complex plane as an entire function of $s$.

Denote by $N_{f}(\sigma, T)$ the number of zeros $\rho = \beta + it$ of $L(f,s)$ in the rectangle $\sigma \leq \beta \leq 1$, $|t| \leq T$. In 1989, Ivi\'{c} \cite{Ivic1992} showed that $N_{f}(\sigma, T) \ll T^{2(1-\sigma)+\varepsilon}$ holds for $\sigma \geq 53/60$. The successful establishment of the density hypothesis in the range $\sigma \geq 1/2$ should be one of the key
 ingredients for obtaining estimates for the asymptotic distribution of $\lambda_{f}(p)$ for primes in short intervals that are as good as if one were to assume the grand Riemann hypothesis. 
 
The main goal of this paper is to improve Ivi\'{c}'s result by showing\footnote{The argument presented here also allows one to provide a range of validity for the density hypothesis for other types of $L$-functions, provided the $L$-function admits an Euler product-type identity, satisfies the Ramanujan conjecture for its coefficients, has an analogue for Good's second moment estimate and admits polynomial bounds in the desired half-plane.}:
\begin{theorem} \label{main1} We have the bound 
	\begin{equation*}
		N_{f}(\sigma, T) \ll_{f, \varepsilon} T^{2(1-\sigma)+\varepsilon}
	\end{equation*}
for $\sigma \geq 1407/1601$.
\end{theorem}
Here $1407/1601\approx 0.8788$, while $53/60= 0.8833...$. 
We now describe the general strategy for our proof of Theorem \ref{main1}. The first step, which is standard, is to apply the zero-detection method to divide the zeros of $L(s,f)$ into two categories, the so-called class-I zeros and class-II zeros. The number of class-II zeros is directly estimated by using Good's second moment estimate for $L$-functions associated with holomorphic cusp forms \cite{Good}.  The innovation of our work is to achieve sharper estimates for the class-I zeros than those obtained by Ivi\'{c} in \cite{Ivic1992}.

We seek to obtain an upper bound for the class-I zeros by applying the Hal\'{a}sz-Montgomery inequality. Ivi\'{c}'s argument is then to combine Huxley's subdivision technique with the direct insertion of mixed moment bounds for the zeta function into this inequality (cf. Remark \ref{remark Ivic method}). Our improvement is based on two aspects. The first one is an optimization of the parameters in the mixed moment estimates. Here we rely on the newly established \emph{exponent pair} $(13/84 + \varepsilon, 55/84 + \varepsilon)$ due to Bourgain \cite{Bourgain2016}. The second aspect that leads to an additional improvement is the incorporation of a dichotomy technique developed by Bourgain in \cite{Bourgain2000} to achieve the current record of the range of validity of the density hypothesis for the Riemann zeta function. The crucial point of the dichotomy is that it allows one to apply Heath-Brown's estimate on double zeta sums \cite{HBLV} which is more efficient than the mixed moment bounds in certain ranges. In Bourgain's original paper the dichotomy approach is a bit difficult to follow; one of the goals of this paper is to explain the underlying ideas and its advantages more clearly. 

In addition to improving the range of validity of the density hypothesis for the zeros of the $L$-functions associated with holomorphic cusp forms, our argument can, with only mild adjustments, also be applied to obtain a zero-density estimate for the Riemann zeta function. In order to further demonstrate the strength of the dichotomy method we prove:
\begin{theorem} \label{main2} There holds
	\begin{equation*}
		N(\sigma, T) \ll_{\varepsilon} T^{\frac{24(1-\sigma)}{30\sigma -11}+\varepsilon}
	\end{equation*}
	for $279/314 \leq \sigma \leq 17/18$.
\end{theorem}
This improves on the condition $155/174 \leq \sigma \leq 17/18$ obtained by Ivi\'{c} (\cite{Ivic1980}; \cite[Thm. 11.2]{Ivicbook}) in 1980. Observe that $155/174\approx0.8908$ and $279/314\approx0.8885$.

The paper is organized as follows. In Section \ref{zedete} we recall the classical zero-detection method to divide the zeros into class-I zeros and class-II zeros and explain how to handle the class-II zeros.
 In Section \ref{Ivices} we revisit Ivi\'{c}'s original argument involving mixed moment estimates. In Section \ref{BouD} we study Bourgain's dichotomy technique in this context;  we derive a large value estimate for Dirichlet polynomials from which Theorem \ref{main1} follows.
 Finally, the proof of Theorem \ref{main2} will be completed in Appendix \ref{eszeta}.

We adopt the convention that $\varepsilon$ stands for a small positive quantity. Throughout the paper, we allow $\varepsilon$ to change by at most a constant factor on places that we do not always specify. We let $\chi_E$ denote the indicator function of a set $E$. We use $\ll$ and $\gg$ to denote Vinogradov's notation, while implied constants depend at most on $\varepsilon$ and the cusp form $f$.

The authors wish to thank Harold Diamond and J\'anos Pintz for providing access to the articles \cite{Bourgain2002} and  \cite{Ivic1992}, respectively. Furthermore, we would also like to thank Olivier Robert for his swift reply clarifying that his result \cite[Thm. 1]{Robert} does not imply that $(1/12,3/4)$ is an exponent pair.

\section{The zero-detection method}\label{zedete}

Our starting point is a zero-detection method which has become standard by now. As our further analysis heavily uses the concepts that are introduced by this method, we opt, for the convenience of the reader, to briefly recall here the main ideas involved in this technique.

Let $X,Y, T > 1$. We consider an approximate inverse for $L(s,f)$ given by
\[
M_{X}(s,f)= \sum_{n \leq X} \frac{\mu_{f}(n)}{n^{s}}.
\]
where $\mu_{f}(n)$ is the multiplicative function for which 
\begin{align*}  \label{muf} \mu_{f}(p^{k}):=\left\{
	\begin{array}{rcl}
		1,     &      &\mbox{if} \ \ k=0,2,       \\
		-\lambda_{f}(p),               &      &\mbox{if}\ \  k=1,    \\
		0,     &      &\mbox{if} \ \ k\geq 3.      \\
	\end{array} \right.
\end{align*}
This gives
$$ L(s,f) M_X(s,f) = \sum_{n=1}^{\infty} c_n n^{-s}, \ \ \ \Re s > 1,
$$
where $c_n = \sum_{d | n , d \leq X} \mu_f(d) \lambda_f(n/d)$. Observe that $c_1 = 1$, $c_n = 0$ for $1 < n \leq X$ and $|c_n| \ll n^{\varepsilon}$.

Introducing the weight $e^{-n/Y}$ and exploiting the Mellin inversion formula for $e^{-x}$, one finds, for $1/2 <\Re s < 1$,
\begin{align*}
  e^{-1/Y} + \sum_{n > X} c_n n^{-s} e^{-n/Y} 
 & = \frac{1}{2\pi i} \int^{2 + i \infty}_{2 - i \infty} \Gamma(z) Y^{z} L(s + z,f) M_X(s+z,f) \mathrm{d}z 
 \\
 & = \frac{1}{2\pi i} \int^{1/2 - \Re s + i \infty}_{1/2 - \Re s - i \infty} \Gamma(z) Y^{z} L(s + z,f) M_X(s+z,f) \mathrm{d}z \\
 & \phantom{\frac{1}{2\pi i} \int^{1/2 - \Re s + i \infty}_{1/2 - \Re s - i \infty} } {}+ L(s,f)M_X(s,f), 
\end{align*}
where we picked up the residue at $z = 0$ while shifting the line of integration\footnote{The function $L(s,f)$ is polynomially bounded on vertical strips \cite[Cor.~3, p.~334]{Good1975}, which justifies the switching of contour.}. The tail of the sum $\sum_{n > Y\log^2 Y} c_n n^{-s} e^{-n/Y}$ is $o(1)$ as $Y\rightarrow \infty$. If $|\Im s| \leq T$, then also the tails $|\Im z| \geq \log^{2} T$ of the final integral become $o(1)$ as $T \rightarrow \infty$ if $X$ is polynomially bounded in $T$, say, in view of the exponential decay on vertical lines of the $\Gamma$ function and the trivial estimate $M_X(1/2 + iu) \ll X^{1/2 + \varepsilon}$. Therefore, for $Y, T$ sufficiently large, we have
\begin{align*}
&L(s,f)M_X(s,f)  = 1 +  \sum_{X < n \leq Y \log^2 Y } c_n n^{-s} e^{-n/Y} + o(1) \\
&  \ \ \ + \int^{\log^2 T}_{- \log^2 T} \Gamma\left(\frac{1}{2} - \Re s + iu\right) Y^{\frac{1}{2} - \Re s + iu} L\left(\frac{1}{2} + it + iu,f\right)M_X\left(\frac{1}{2} + it + iu,f\right) \mathrm{d}u .
\end{align*}

Thus, if $\rho = \beta + it$ is a zero of $L(s,f)$ with $1/2 < \beta < 1$ and $|t| \leq T$, then either
\begin{equation} \label{I-ze}
\left| \sum_{X< n \leq Y\log^{2}Y}c_n n^{-\rho}e^{-n/Y}\right| \gg 1, 
\end{equation}
or
\begin{equation} \label{II-ze}
	\left| \int_{-\log ^{2} T}^{\log ^{2} T} L\left(\frac{1}{2}+i(t +u), f\right) M_{X}\left(\frac{1}{2}+i (t+u),f \right) Y^{\frac{1}{2}-\beta+iu}\Gamma\left(\frac{1}{2}-\beta+iu\right)\:\mathrm{d}u \right| \gg 1.
\end{equation}
The zeros $\rho$ with $\beta \geq \sigma$ and $|t| \leq T$ for which \eqref{I-ze} holds are referred to as \emph{class-I zeros} while those for which \eqref{II-ze} holds are called \emph{class-II zeros}. As a zero must inevitably belong to one (or both) of these classes we obtain, for $1/2 < \sigma < 1$,
\begin{equation} \label{NOze}
	N_{f}(\sigma, T) \ll (|R_1|+|R_2|)T^{\varepsilon},
\end{equation}	
where $R_{1} = R_1(X,Y,T)$, resp. $R_2$, is the set of class-I, resp. class-II zeros, and $|R_{j}|$ denotes their cardinality.

For both of these classes we now consider a (saturated) subset $\tilde{R}_{j}$ of \emph{well-spaced zeros}; those are subsets of $R_j$ for which the imaginary parts of the zeros are well-spaced in the sense that 
\begin{equation} \label{disze}
	|t_{1} - t_{2}| \geq T^{\varepsilon}, 
\end{equation}
for different $\rho_1 = \beta_1 + it_1$ and $\rho_2 = \beta_2 + it_2$ belonging to $\tilde{R}_j$. Since  $N_f(\sigma, T+ 1) - N_f(\sigma, T) \ll \log T$, as follows from e.g. \cite[Thm.~3.5, p.~156]{Moreno1977}, one can always select a set of well-spaced zeros $\tilde{R}_j$ such that $|\tilde{R}_j| \gg |R_j| / (T^\varepsilon \log T)$. Therefore, the estimate \eqref{NOze} remains valid if we replace $R_j$ by $\tilde{R}_j$.  

\subsection{The contribution of the class-II zeros}

We begin by analyzing the contribution of the well-spaced class-II zeros $\tilde{R}_{2}$. If we set 
\begin{equation*} 
\left|L\left(\frac{1}{2}+i\gamma_r,f\right)\right|=\max_{-\log ^{2} T\leq u \leq \log ^{2} T}\left|L\left(\frac{1}{2}+it_{r} +iu,f\right)\right|,
\end{equation*}
where $t_r$ are the imaginary parts of the class-II zeros, then we find
\begin{equation*} 1 \ll T^{\varepsilon} Y^{1/2-\sigma} \left|L\left(\frac{1}{2}+i\gamma_{r},f\right)\right|,\ \ \ r=1,2, ... ,|\tilde{R}_{2}|,
\end{equation*}
where we have set $X = T^\varepsilon$. We square the above inequality and as the $\gamma_r$ are well-spaced because the class-II zeros are, we may apply\footnote{Good's mean value theorem yields $\int_{0}^{T}|L(1/2+it,f)|^{2}\mathrm{d}t\ll T\log T$. The sum version that we use here can be derived from the integral form along the same lines as it is done for the Riemann zeta function, cf. argument in \cite[p.~200]{Ivicbook}. Alternatively, one may derive a second moment estimate on $L'(s,f)$ and apply Gallagher's lemma, as is e.g. done in \cite{Xu}.} Good's second moment estimate \cite{Good}  to obtain 
\begin{equation*} 
	|\tilde{R}_{2}| \ll T^{\varepsilon} Y^{1-2 \sigma} \sum_{r \leq |\tilde{R}_{2}|} \left|L\left(\frac{1}{2}+i\gamma_{r},f\right)\right|^{2} \ll T^{1+\varepsilon}Y^{1-2\sigma}.
\end{equation*} 
Therefore, upon choosing $Y=T$, we obtain $|\tilde{R}_2| \ll T^{2(1-\sigma) + \varepsilon}$ and this already concludes the analysis of the class-II zeros.

\subsection{Representative class-I zeros}

The rest of the argument is then to bound the contribution of the class-I zeros. First we shall restrict the well-spaced class-I zeros even further. By a dyadic subdivision of the interval $(X, Y \log^2 Y]$, one can find $X \leq M < Y\log^2 Y$ such that 
\begin{equation} \label{proze}
\left| \sum_{M< n \leq 2M}c_n n^{-\rho}e^{-n/Y} \right| \gg \frac{1}{\log Y}
\end{equation}
for at least  $|\tilde{R}_1|\log 2/\log (Y\log^2 Y)$ zeros $\rho$ of $\tilde{R}_1$. The elements of $\tilde{R}_1$ that additionally satisfy \eqref{proze} are called \emph{representative well-spaced zeros} and this subset will be denoted as $R$. We remark that \eqref{NOze} remains valid upon replacing $|R_1|$ by $|R|$.

Next, we are going to find some very useful estimates allowing us to bound the size of a set of representative well-spaced zeros in terms of the moduli of certain Dirichlet polynomials. It also turns out that the most problematic range is when $M$ is small; the following argument shall allow us to take care of the range $M < T^{1/2}$ such that the critical range for $M$ becomes $M \approx T^{1/2}$.

Let $\nu$ be a fixed integer and let $A$ be a multiset consisting of elements of $R$. We shall actually set $\nu = 1$ in the proof of Theorem \ref{main1} and $\nu = 2$ for Theorem \ref{main2}. We consider an integer power $k$ such that $M^{k} \leq Y^{\nu+\varepsilon} <M^{k+1}$. Hence, as we have set $X = T^{\varepsilon}$ and $Y = T$, we deduce $1 \leq \nu \leq k \ll_{\varepsilon} 1$ and $Y^{\frac{\nu^{2}}{\nu+1}}<M^{k}\leq Y^{\nu+\varepsilon}$. Raising \eqref{proze} to the power $k$ we find
\begin{equation*}
	\left| \sum_{M^{k}< n \leq 2^{k}M^{k}}c_n'n^{-\rho} \right| \gg (1/ \log Y)^{k},
\end{equation*}
where $c_n' = \sum_{n_1n_2 \dots n_k = n} c_{n_1} c_{n_2} \dots c_{n_k}e^{-\frac{n_{1}+\cdot\cdot\cdot+n_{k}}{Y}}$. Again $c_n' \ll n^{\varepsilon}$. A dyadic subdivision of the interval $(M^k, 2^k M^k]$ allows us to find $M^{k}\leq N = N(A) \leq 2^{k}M^{k}$ for which
\begin{align*}
|A| \ll & (\log Y)^{k} \sum_{\rho \in A} \left| \sum_{N < n \leq 2N}c_n'n^{-\rho} \right|
=  (\log Y)^{k} \sum_{\rho \in A} \left| \int_{N}^{2N} u^{-\beta}  \mathrm{d} \left(\sum_{N< n\leq u}c_n'n^{-it}\right) \right|
\\
= & (\log Y)^{k} \sum_{\rho \in A} \left| (2N)^{-\beta}\sum_{N< n \leq 2N} c_n'n^{-it}
+\int_{N}^{2N} \beta u^{-\beta-1} \sum_{N< n\leq u}c_n'n^{-it} \:\mathrm{d} u \right|
\\
\ll &  (\log Y)^{k} N^{-\sigma} \max_{N< u \leq 2N}\sum_{\rho \in A} \left|\sum_{N< n \leq u} c_n'n^{-it} \right|.
\end{align*}
If we let $c_n'' = 0$ after the point where the above maximum is reached, but $c_n'' = c_n'$ otherwise, we obtain
\begin{equation*} \label{uppA}
|A| \ll N^{-\sigma +\varepsilon} \sum_{\rho \in A} \left|\sum_{N< n \leq 2N} c_n''n^{-it}\right|, \ \ \ |A| \ll N^{-2 \sigma +\varepsilon} \sum_{\rho \in A} \left|\sum_{N< n \leq 2N} c_n''n^{-it}\right|^{2},
\end{equation*}
the last inequality being derived from Cauchy-Schwarz. If we now set $b(n) = b(n,A) = \epsilon c_n''$ for a sufficiently small $\epsilon$ such that $|b(n)| \leq 1$, we have
\begin{equation} \label{uppAF}
	|A| \ll N^{-\sigma +\varepsilon} \sum_{\rho \in A} \left|\sum_{N< n \leq 2N} b(n)n^{-it}\right|,\ \ \ |A| \ll N^{-2\sigma +\varepsilon} \sum_{\rho \in A} \left|\sum_{N< n \leq 2N} b(n)n^{-it}\right|^{2}.
\end{equation}  
In particular, if one selects $A=R$, we get 
\begin{equation} \label{uppR1'F}
	|R| \ll N^{-\sigma +\varepsilon} \sum_{\rho \in R} \left|\sum_{N< n \leq 2N} b(n)n^{-it}\right|.
\end{equation}

If one were to trivially estimate the right-hand side, one would obtain the bound $N^{1-\sigma + \varepsilon} |R|$ and this delivers no information at all as it is way worse than the trivial bound $|R|$. Our goal in the next section is therefore to realize there is indeed sufficient cancellation in \eqref{uppR1'F} to extract some non-trivial information. 

We do emphasize here again that $N$ and $b(n)$ do depend on the set $A$. Throughout the rest of the paper, we shall write $b(n)$ and $N$ when we refer to the set $R$. If any other set of representative well-spaced zeros $A$ is considered, we shall explicitly mention the dependence of $b(n)$ and $N$ on $A$. On the other hand we note that $M^k \leq N(A) < 2^k M^k$, therefore $N \ll N(A) \ll N$ for any $A$. Furthermore, as we take $\nu = 1$, one has 
\begin{equation} \label{LengthPolyF}
T^{1/2} \leq N < T^{1+\varepsilon}.
\end{equation}

\section{Ivi\'{c}'s Estimate}\label{Ivices}

In this section we deduce a first non-trivial estimate on the number of class-I zeros. The first step is to apply the Hal\'{a}sz-Montgomery inequality to realize there is cancellation in \eqref{uppR1'F}. The following lemma is a reformulation of the estimate in \cite[Eq. 11.40]{Ivicbook}. We closely follow here the proof of \cite[Thm.  11.2]{Ivicbook}.

\begin{lemma} \label{reIv} Let $A \subseteq R$ be a set of representative well-spaced class-I zeros (where we do not allow repetition). For $\ell \in \mathbb {Z}$, define
\begin{equation*} 
	\Delta_{A}(\ell) =\# \{ (\beta+it, \beta'+it') \in A \times A  :     |t-t'-\ell| <1 \}.
\end{equation*}
Then
	\begin{equation} \label{conreIv}
		|A| \ll \left\{N^{2-2\sigma} +N^{3/4-\sigma}\left[ \sum_{\ell \in \mathbb {Z}} \Delta_A (\ell) \int_{-2 \log^{2} T}^{2 \log^{2} T}  \left| \zeta \left(\frac{1}{2}+ iv + i\ell \right)\right|  \: \mathrm{d} v \right]^{\frac{1}{2}}\right\} T^{\varepsilon}.
	\end{equation}
\end{lemma}

\begin{proof} As $N \asymp N(A)$, the estimate \eqref{conreIv} is equivalent upon replacing $N$ with $N(A)$. Throughout the rest of the proof, however, we shall simply write $N$ for $N(A)$ in order not to overload the notation unnecessarily. 

By applying the Hal\'{a}sz-Montgomery inequality \cite[Lemma 1.7, p. 6]{MontgomeryBook} to \eqref{uppAF}, we get
\begin{equation} \label{AHM}
	|A| ^{2}N^{2\sigma -2\varepsilon} \ll |A| N^{2} + N\sum_{r \neq s}\left| H (it_{r}-it_{s})\right|,
\end{equation}
where $t_r, t_s$ denote the imaginary parts of elements of $A$ and
\begin{equation*} 
	H(it)= \sum_{n=1}^{\infty} (e^{-n/2N}-e^{-n/N})n^{-it}
	=  \frac{1}{2\pi i} \int_{2-i \infty}^{2+i \infty} \zeta\left(w+ it\right) \Gamma(w)\left( (2N)^{w} -N^{w}\right) \: \mathrm{d}w.
\end{equation*}

We switch the contour to the line $\Re w= 1/2$ which is allowed since $\zeta$ is polynomially bounded and $\Gamma$ decays exponentially on vertical strips. We pass over a simple pole at $w=1-it$ with residue $O(Ne^{-|t|})$ and our equation becomes
\begin{equation*} 
	H(it)
	=  \frac{1}{2\pi i} \int_{1/2-i \infty}^{1/2+i \infty} \zeta\left(w+ it\right) \Gamma(w)\left( (2N)^{w} -N^{w}\right) \: \mathrm{d}w + O\left(Ne^{-|t|}\right).
\end{equation*}

When $|\Im w| \geq \log^{2}T$, the tails of the integral contribute $O(N^{1/2} \exp(-\log^2 T)) = o(1)$ as $N \ll T^{1+\varepsilon}$. Therefore, 

\begin{align*} 
	\sum_{r \neq s} \left|H(it_{r}-it_{s})\right| & \ll N\sum_{r \neq s} e^{-|t_{r}-t_{s}|}+ o\left(|A|^{2}\right)
	\\
	& \ \ \ + N^{1/2} \int_{-\log^{2}T}^{\log^{2}T} \sum_{r \neq s} \left|\zeta\left(\frac{1}{2} +it_{r}-it_{s}+iv\right)\right|  \: \mathrm{d}v.
\end{align*}
The first term on the right hand side is $o\left(|A|^{2}\right)$ as the members of $A$ are well-spaced. Moreover, by the definition of $\Delta_A (\ell)$
, we have
\begin{align*} 
	&\int_{-\log^{2}T}^{\log^{2}T} \sum_{r \neq s} \left|\zeta\left(\frac{1}{2} +it_{r}-it_{s}+iv\right)\right|  \: \mathrm{d}v
	\ll & \sum_{\ell \in \mathbb {Z}} \Delta_A (\ell) \int_{-1-\log^{2}T}^{1+\log^{2}T} \left|\zeta\left(\frac{1}{2} +i\ell+iv\right)\right|  \: \mathrm{d}v.
\end{align*}
As $\sigma > 1/2$ we arrive at \eqref{conreIv} after inserting all these estimates in \eqref{AHM}.
\end{proof}
As usual if we do not mention the subscript $A$ for $\Delta_A$ when we are referring to $A = R$.

It thus remains to find adequate estimates for the integral in \eqref{conreIv}. For this we shall appeal to moment estimates on the zeta function. Let $B_0, B_1, q_0, q_1$ be positive numbers for which $q_0, q_1 \geq 2$, and $H : [0,\infty) \rightarrow [1,\infty)$. We let 
$$\zeta_0 = \zeta_{0,T} = \zeta \chi_{\{s : |\zeta(s)|\geq H(T)\}}, \ \  \zeta_1 = \zeta_{1,T} =\zeta \chi_{\{s : |\zeta(s)|< H(T)\}}, $$
where $\chi$ denotes the characteristic function of a set. In what follows, we rely on an assumption of the form
\begin{equation} \label{mixmo}
	\int_{0}^{T}\left|\zeta_0\left(\frac{1}{2} + it\right)\right|^{q_{0}}\mathrm{d}t \ll T^{B_{0}}, \ \ \ \int_{0}^{T}\left|\zeta_1\left(\frac{1}{2} + it\right)\right|^{q_{1}}\mathrm{d}t \ll T^{B_{1}}. 
\end{equation}
involving mixed moment estimates  for the zeta function. 

Recall $\Delta (\ell) =\# \{ (\beta+it, \beta'+it') \in R \times R  :     |t-t'-\ell| <1 \}.$ We have $\Delta (\ell) \leq |R|$ because the elements of $R$ are well-spaced and $\sum_{\ell \in \mathbb {Z}} \Delta (\ell) \leq 2|R|^{2}$ as each couple $(\beta + it, \beta + it')$ can at most contribute to two $\Delta(\ell)$. It follows that for $q > 1$,
\begin{equation*} \label{diset1}
	\sum_{\ell \in \mathbb {Z}} \Delta (\ell)^{\frac{q}{q-1}} \leq \sum_{\ell \in \mathbb {Z}}\Delta (\ell) \Delta (\ell)^{\frac{1}{q-1}} \leq2|R|^{\frac{2q-1}{q-1}}. 
\end{equation*}
With this estimate the integral in \eqref{conreIv} becomes through some applications of H\"older's inequality
\begin{align*} 
	& \sum_{\ell \in \mathbb {Z}} \Delta (\ell) \int_{-2\log^{2}T}^{2\log^{2}T} \left|\zeta\left(\frac{1}{2} +i\ell+iv\right)\right|  \: \mathrm{d}v 
	\\
	 & \ \ \ \leq  \sum_{j = 0}^{1} \left( \sum_{\ell \in \mathbb {Z}} \Delta (\ell)^{\frac{q_j}{q_j - 1}} \right)^{\frac{q_j - 1}{q_j}} \left\{ \sum_{\ell \in \mathbb {Z}}  \left[\int_{-2\log^{2}T}^{2\log^{2}T} \left|\zeta_{j}\left(\frac{1}{2} +i\ell+iv\right)\right|^{q_j}  \: \mathrm{d}v\right] \right\}^{\frac{1}{q_j}}T^{\varepsilon} 
	\\
	& \ \ \ \ll  |R|^{2 - \frac{1}{q_0}} T^{\frac{B_0}{q_0}+\varepsilon} + |R|^{2 - \frac{1}{q_1}} T^{\frac{B_1}{q_1}+\varepsilon},
	\end{align*}
and thus
    \begin{align*}
    	|R| \ll T^{\varepsilon}\left(N^{2-2\sigma} +T^{B_0}N^{(3-4\sigma) q_0/2}+T^{B_1}N^{(3-4\sigma)q_1/2}\right).
    \end{align*}

We now refine this estimate through Huxley's subdivision argument.
Set
\begin{equation} \label{Hdi}
	T_{0}=\delta_{2}T,\ \ \ \ T^{-1} \leq \delta_{2} \leq 1, 
\end{equation}
and let $R_{\alpha}= \{ \rho \in R : \Im \rho \in I_{\alpha}\}$, with $I_{\alpha} \subset [-T, T ]$ a certain subinterval of length $T_{0}$. 
Repeating the above argument with $R_\alpha$ instead of $R$ then gives\footnote{The subdivision of $\zeta$ into $\zeta_0$ and $\zeta_1$ is now with respect to $H(T_0)$ instead of $H(T)$.}
\begin{align*} 
	|R_{\alpha}| \ll T^{\varepsilon}\left(N^{2-2\sigma} +\delta_{2}^{B_0}T^{B_0}N^{(3-4\sigma)q_0/2}+\delta_{2}^{B_1}T^{B_1}N^{(3-4\sigma)q_1/2}\right).
\end{align*}
Subdividing $R$ into about as many as $\lceil 2/\delta_{2}\rceil$ sets of the form $R_\alpha$ and summing these contributions gives
\begin{align*}
	|R| \ll T^{\varepsilon}\left(\delta_{2}^{-1}N^{2-2\sigma} +\delta_2^{B_0 - 1}T^{B_0}N^{(3-4\sigma)q_0/2}+\delta_{2}^{B_1 - 1}T^{B_1}N^{(3-4\sigma)q_1/2}\right).
\end{align*}
 Now, selecting $\delta_{2}=N^{2-2\sigma}T^{2\sigma-2}$ if $N \leq T$ and $\delta_2 = 1$ if $T < N < T^{1+\varepsilon}$ delivers\footnote{One can optimize the choice for $\delta_2$ even further here. However, we are in this work only interested into which range the density hypothesis is valid. If one only considers this question there is no advantage in further optimizing $\delta_2$.}
\begin{align*} 
	|R| \ll T^{2-2\sigma + \varepsilon} + \sum_{j= 0}^{1}T^{(2B_j - 2)\sigma + (2-B_j) + \varepsilon} N^{(2B_j -2 +3q_j/2) - 2\sigma(B_j + q_j -1)}.
\end{align*}
We thus obtain $|R| \ll T^{2(1-\sigma)+\varepsilon}$ provided $\sigma \geq 1 - q_j/(4B_j + 4q_j -4)$ for $j = 0,1$ and provided $N$ satisfies 
\begin{equation} \label{IvrelenN0}
	N \geq T^{\frac{4B^{\ast}\sigma - 2B^{\ast}}{(4q^{\ast} +4B^{\ast} - 4)\sigma - (3q^{\ast} + 4B^{\ast} -4)}},
\end{equation}
where $(q^{\ast},B^{\ast})$ is the couple $(q_j,B_j)$, $j = 0,1$, for which the exponent above is maximal, which in principle may depend on $\sigma$. If $\sigma$ lies in the range where this exponent is less than $1/2$, that is $\sigma \geq (3q^{\ast}-4)/(4q^{\ast}-4B^{\ast}-4)$, we are done as we always have $N \geq T^{1/2}$. 
In the remaining range $\sigma < (3q^{\ast}-4)/(4q^{\ast}-4B^{\ast}-4)$, we may therefore assume in the sequel that
\begin{equation} \label{IvrelenN}
	T^{1/2} \leq N < T^{\frac{4B^{\ast}\sigma - 2B^{\ast}}{(4q^{\ast} +4B^{\ast} - 4)\sigma - (3q^{\ast} + 4B^{\ast} -4)}}.
\end{equation}

\begin{remark}\label{remark Ivic method} We briefly mention how Ivi\'c arrived at the validity of the density hypothesis in the range $\sigma \geq 53/60$. He selected $q_0 =6, q_1 = 19, B_0 = 1 +\varepsilon$ and $B_1 = 3 + \varepsilon$ with $H(T) = T^{2/13}$. The validity of \eqref{mixmo} under these parameters is guaranteed by \cite[Cor. 8.1, Eq. 8.31]{Ivicbook}. With this choice the condition \eqref{IvrelenN0} becomes $N \geq \max \{T^{6(2\sigma - 1)/(84 \sigma - 65)}, T^{(2\sigma -1)/(12\sigma - 9)} \}$ and this is always satisfied if $\sigma \geq 53/60$ as the exponents of $T$ are then always smaller than $1/2$.
\end{remark}

\section{Bourgain's dichotomy 
}\label{BouD}

In this section, inspired by the work of Bourgain \cite{Bourgain2000, Bourgain2002}, we will use the dichotomy method to improve the estimates for the integral terms appearing in \eqref{conreIv}. This allows us to obtain a new estimate for the class-I zeros.

\subsection{Lemmas on Dirichlet polynomials}

\

In applying Bourgain's method, we shall require some preliminary lemmas on estimations for Dirichlet polynomials. The first lemma gives an upper estimate for pointwise values of a Dirichlet polynomial in terms of an average of the values near the point. It is a slight modification of \cite[Lemma 4.48]{Bourgain2000}.
\begin{lemma} \label{Bourgpoly} Consider the Dirichlet polynomial
	\begin{equation*} 
		F(t)=\sum_{N < n \leq 2N} b_{n} n^{-it},\ \ \ t \in \mathbb{R},
	\end{equation*}
	where the coefficients $b_n$ satisfy $|b_{n}| \leq 1$. Then,
	\begin{equation*} 
		|F(t)| \ll 1 +  \log N \int_{|v|<\log N} |F(t + v)| \:\mathrm{d}v, \ \ \ N\rightarrow \infty.
	\end{equation*}
\end{lemma}
\begin{proof} Let $\psi$ be a smooth function on $\mathbb{R}$ such that $\hat{\psi}$, the Fourier transform of $\psi$, is identically $1$ on the interval $[1,2]$ and satisfies 

\begin{equation*} 
|\psi (x)| \ll  e^{-|x|^{2/3}}. 
\end{equation*}
The existence of such a function $\psi$ is guaranteed by the Denjoy-Carleman theorem.

Let $\psi_{\lambda}(x)=(1/\lambda)\psi (x/ \lambda)$. We have, for $N\geq 2$,  
\begin{align*} 
	|F(t)| =& \left| \sum_{N < n \leq 2N} b_{n} \hat{\psi} \left( \frac{\log n}{\log N}\right) n^{-it} \right|
	= \left| \sum_{N < n \leq 2N} b_{n}  \hat{\psi}_{(\log N)^{-1}} \left( \log n \right) n^{-it} \right|
	\\
	=& \left|  \int_{\mathbb{R}} F(t+v) \psi_{(\log N)^{-1}}(v) \:\mathrm{d}v  \right|.
\end{align*}

The result now follows upon realizing that $|\psi_{(\log N)^{-1}}(v)| \ll \log N$  if   $|v| < \log N$ and
\begin{equation*} 
	 \int_{|v| \geq \log N} \left|  F(t+v) \psi_{(\log N)^{-1}}(v) \right|  \:\mathrm{d}v
	   \ll N \log N\int_{|v| \geq \log N} e^{-(|v| \log N)^{2/3}}  \:\mathrm{d}v
	 \ll 1, 
\end{equation*}
as $N \rightarrow \infty$ because $|b_n| \leq 1$.
\end{proof}
The next one is a simple estimate due to Bourgain \cite[Lemma 3.4]{Bourgain2000} for Dirichlet polynomials over difference sets where the index sets are different.
 \begin{lemma} \label{BJ} Let $a_{n},b_{n}$  $(1 \leq n \leq N)$ be complex numbers such that $|a_{n}| \leq  b_{n}$. 
 Let $R, S \subseteq \mathbb{R}$ be two finite sets. Then 
\begin{equation*} 
\sum_{\substack{t \in R \\s \in S}} \left| \sum_{n=1}^{N} a_{n} n^{i(t-s)}\right|^{2} 
\leq  \left(\sum_{t, t' \in R} \left| \sum_{n=1}^{N} b_{n} n^{i(t-t')}\right|^{2} \right)^{\frac{1}{2}} \left(\sum_{s, s' \in S} \left| \sum_{n=1}^{N} b_{n} n^{i(s-s')}\right|^{2} \right)^{\frac{1}{2}} .
\end{equation*}
\end{lemma}

The final lemma is Heath-Brown's estimate on double zeta sums \cite[Thm.  1]{HBLV} (cf. \cite[Lemma 3.7]{Bourgain2000}). It is much deeper and is a crucial ingredient for our argument.  
\begin{lemma} \label{Headou} Let $R$ be a finite set of well-spaced, cf. \eqref{disze}, points such that $|t| \leq T$ for each $t\in R$. 
Then
\begin{equation*} 
\sum_{t,t'\in R} \left| \sum_{N < n \leq 2N} n^{i(t-t')}\right|^{2} \ll T^{\varepsilon}(|R|N^{2}+ |R|^{2}N+N|R|^{5/4}T^{1/2}).
\end{equation*}
\end{lemma}

\subsection{The dichotomy}\label{The Dich}
\

Let $T_{0}$ and $R_{\alpha}$ be as in Section \ref{Ivices}, see \eqref{Hdi}. We recall that we have set 
\begin{equation*} 
	\Delta_{\alpha} (\ell) = \Delta_{R_{\alpha}}(\ell) = \# \{ (\beta+it, \beta'+it') \in R_{\alpha} \times R_{\alpha}  :     |t-t'-\ell| <1 \}.
\end{equation*}
Let $0< \delta_{1} < 1$ be a parameter to be optimized later. We set
$$\zeta_0 = \zeta_{0,T_0, H} = \zeta \chi_{\{s : |\zeta(s)|\geq H(T)\}}, \ \  \zeta_1 = \zeta_{1,T_0,H} =\zeta \chi_{\{s : 1\leq|\zeta(s)|< H(T)\}}, $$ and let $B_0, B_1, q_0, q_1$ be the parameters that were introduced in the mixed moment estimates \eqref{mixmo}. Note that the definition for $\zeta_1$ is slightly different than in the previous section because of the lower bound $|\zeta| \ge 1$. For each fixed $\alpha$, we distinguish between the following alternatives.
\vspace{3px}
\\
\textbf{Case 1.} We have
\begin{equation} \label{Dich1}
    \sum_{\ell \in \mathbb {Z}} \Delta_{\alpha} (\ell) \int_{|v| < T^{\varepsilon}} \left|\zeta_{0}\left(\frac{1}{2} +i\ell+iv\right)\right|  \: \mathrm{d}v 
	\leq  \delta_{1}^{2/q_{0}}T_{0}^{B_{0}/q_{0}} |R_{\alpha}|^{2-1/q_{0}} 
\end{equation}
and
\begin{equation} \label{Dich2}
	\sum_{\ell \in \mathbb {Z}} \Delta_{\alpha} (\ell) \int_{|v| < T^{\varepsilon}} \left|\zeta_{1}\left(\frac{1}{2} +i\ell+iv\right)\right|  \: \mathrm{d}v 
	\leq  \delta_{1}^{2/q_{1}}T_{0}^{B_{1}/q_{1}} |R_{\alpha}|^{2-1/q_{1}} .
\end{equation}
\vspace{3px}
\\
\textbf{Case 2.} Either \eqref{Dich1} or \eqref{Dich2} fails.
\vspace{1px}
\newline

We consider a collection of $\lceil2/\delta_{2}\rceil$ sets $R_\alpha$ that cover $R$. We let $\mathcal{I}_0$ be the index set of $\alpha$ for which \eqref{Dich1} fails, $\mathcal{I}_1$ be the index set for which \eqref{Dich2} fails, and $\mathcal{I}_2$ be the index set for which both inequalities hold. Clearly $|R| \ll \sum_{\alpha \in \mathcal{I}_0} R_\alpha + \sum_{\alpha \in \mathcal{I}_1} R_\alpha + \sum_{\alpha \in \mathcal{I}_2} R_\alpha $. An additional constraint on the parameter $\delta_1$ will arise below in the analysis of Case 2.

\subsection{The Case 1 contribution} \label{caseIcon}

\

We suppose here that $|R| \ll \sum_{\alpha \in \mathcal{I}_2} R_\alpha$. Let $\alpha \in \mathcal{I}_2$. Since  $|\zeta| = |\zeta_{0}| + |\zeta_{1}| + |\zeta\chi_{\{s : |\zeta(s)| < 1\}}|$, it follows that
\begin{align*}
	 \sum_{\ell \in \mathbb {Z}} \Delta_{\alpha} (\ell) \int_{|v| < T^{\varepsilon}} \left|\zeta\left(\frac{1}{2} +i\ell+iv\right)\right|  \: \mathrm{d}v 
	& \leq  \delta_{1}^{2/q_{0}}T_{0}^{B_{0}/q_{0}} |R_{\alpha}|^{2-1/q_{0}} \\
	& \ \ \ + \delta_{1}^{2/q_{1}}T_{0}^{B_{1}/q_{1}} |R_{\alpha}|^{2-1/q_{1}} + |R_\alpha|^2 T^{\varepsilon},
\end{align*}
where the last term is coming from the contribution of $|\zeta| < 1$. Inserting this estimate in 
 \eqref{conreIv} and rearranging $|R_\alpha|$ gives\footnote{Note that the term $|R_\alpha|^2$ can never be dominant if $\sigma > 3/4$, as \eqref{conreIv} would then entail $|R_\alpha| \ll N^{3/4 - \sigma} |R_\alpha| T^{\varepsilon}$, which is impossible in view of $N \geq T^{1/2}$.}, for $\sigma > 3/4$, 
 \begin{align*} 
 	|R_{\alpha}| \ll T^{\varepsilon}\left(N^{2-2\sigma} +\delta_{1}^{2}T_{0}^{B_0}N^{(3-4\sigma)q_0/2} + \delta_{1}^{2}T_{0}^{B_1}N^{(3-4\sigma)q_1/2}\right).
 \end{align*}
Replacing $T_0$ by $\delta_2 T$ and summing over the index set $\mathcal{I}_2$ then yields
 \begin{align} \label{Dichcon4}
 	|R| \ll T^{\varepsilon}\left(\delta_{2}^{-1}N^{2-2\sigma} +\delta_{1}^{2}\delta_{2}^{B_{0}-1}T^{B_{0}}N^{(3-4\sigma)q_{0}/2} + \delta_{1}^{2}\delta_{2}^{B_{1}-1}T^{B_{1}}N^{(3-4\sigma)q_{1}/2}\right).
 \end{align}
\subsection{Analysis of Case 2}\label{AnacaseII}

\

In this section we consider the case when 
\begin{equation} \label{case1fail}
	|R| \ll T^{\varepsilon} \sum_{\alpha \in \mathcal{I}_0} |R_{\alpha}|.
\end{equation}
The analysis of the case when $|R| \ll \sum_{\alpha \in \mathcal{I}_1} |R_{\alpha}|$ is analogous. We have incorporated the extra factor $T^{\varepsilon}$ in \eqref{case1fail} as in some places of the analysis we shall add extra restrictions on the set $\mathcal{I}_0$ and the extra factor $T^{\varepsilon}$ shall guarantee that \eqref{case1fail} remains valid under these restrictions. We write for simplicity $q$ and $B$ instead of $q_{0}$ and $B_{0}$.

First, we translate the failure of \eqref{Dich1} and the dominance of the index set $\mathcal{I}_0$ into a lower bound for the size of a specific multiset of representative well-spaced class-I zeros. In this part we shall perform numerous dyadic decompositions and exploit the mixed moment estimate \eqref{mixmo}. Afterwards we apply the analysis of Section \ref{zedete} to find an upper estimate for this multiset in terms of a Dirichlet polynomial which will subsequently be estimated with the technology provided by Lemma \ref{Headou}. The compatibility of this upper and lower estimate shall then result in an improved estimate on $|R|$.

For $0 < \delta' < 1$ we define the set
\begin{equation*}
	D_{\alpha}(\delta') = \{\ell : \delta' |R_{\alpha}| <  \Delta_{\alpha} (\ell) \leq 2 \delta' |R_{\alpha}| \}.
\end{equation*}
As $\sum_{\ell} \Delta_\alpha(\ell) \leq 2|R_\alpha|^2$ we immediately obtain $|D_{\alpha}(\delta')| \leq 2|R_\alpha|/\delta'$. Furthermore, as $\Delta_{\alpha}(\ell)$ is an integer and $|R_\alpha| \leq T_0$ because the points of $R_\alpha$ are well-spaced and $I_\alpha$ has length at most $T_0$, we may through a dyadic argument find $\delta' \in \{2^{-k} | 1\leq k \ll \log T_0\}$  such that
\begin{align*} 
	\sum_{\ell \in \mathbb {Z}} \Delta_{\alpha} (\ell) \int_{|v| < T^{\varepsilon}} \left|\zeta_{0}\left(\frac{1}{2} +i\ell+iv\right)\right|  \: \mathrm{d}v & 	\ll T_{0}^{\varepsilon} \delta' |R_{\alpha}| \sum_{\ell \in D_{\alpha}(\delta')} \int_{|v| < T^{\varepsilon}} \left|\zeta_{0}\left(\frac{1}{2} +i\ell+iv\right)\right|  \: \mathrm{d}v.
\end{align*}
A priori $\delta'$ does depend on $\alpha$, but as there are only $O(\log T)$ possibilities for $\delta'$, the pigeonhole principle asserts that one may select a subset of $\mathcal{I}_0$, which we shall continue to write as $\mathcal{I}_0$, for which the above expression 
holds for a single $\delta'$ and where \eqref{case1fail} remains valid, possibly with a different value for $\varepsilon$. In conclusion, the parameter $\delta'$ can be chosen independent of $\alpha$.  

Exploiting now that \eqref{Dich1} fails, we obtain
\begin{equation*}
	  \sum_{\ell \in D_{\alpha}(\delta')} \int_{|v| < T^{\varepsilon}} \left|\zeta_{0}\left(\frac{1}{2} +i\ell+iv\right)\right|  \: \mathrm{d}v \geq \delta_{1}^{2/q}(2\delta')^{-1}T_{0}^{B/q-\varepsilon} |R_{\alpha}|^{1-1/q}.
\end{equation*}

Next we proceed to narrow the range for the modulus of $\zeta_0$. For $H > 0$, we consider level sets 
\begin{equation*}
	S_{H,T_0} = S_{H, T_{0}, \varepsilon} = \left\{|t| \leq T_{0}+T^{\varepsilon}+1  :  H < \left|\zeta_{0}\left(\frac{1}{2} +it\right)\right| \leq  2H \right\}.
\end{equation*}
As one can cover the support of $\zeta_0(1/2+it)\chi_{ [-T_{0}-T^{\varepsilon}-1, T_{0}+T^{\varepsilon}+1]}(t)$  by as many as $O(\log T_0)$ level sets of the form $S_{H,T_0}$ in view\footnote{The last inequality $|\zeta_0(1/2+ it)| \ll T_0^{1/4}$ follows from the trivial convexity bound for $\zeta$. Of course there are better estimates available. Here is also why we invoked the additional bound $|\zeta_1| \geq 1$ in the definition of $\zeta_1$. This additional restriction guarantees that also the support of $\zeta_1$ can be covered by $O(\log T_0)$ level sets of the form $S_{H,T_0}$, which is unclear otherwise.} of $1 \leq |\zeta_0(1/2 +it)| \leq T_0^{1/4}$ (on the support of $\zeta_0$), one can find through another dyadic argument a number $H$  such that 
\begin{align*}
		\sum_{\ell \in D_{\alpha}(\delta')} \int_{|v| < T^{\varepsilon}} \left|\zeta_{0}\left(\frac{1}{2} +i\ell+iv\right)\right|  \: \mathrm{d}v & = \sum_{\ell \in D_{\alpha}(\delta')} \int_{-T_{0}-T^{\varepsilon}-1}^{T_{0}+T^{\varepsilon}+1} \left|\zeta_{0}\left(\frac{1}{2} +it\right)\right| \chi_{[\ell-T^\varepsilon, \ell+T^{\varepsilon}]}(t) \mathrm{d}t \\
		& \ll \sum_{\ell \in D_{\alpha}(\delta')} H \log T_0  \int_{-T_{0}-T^{\varepsilon}-1}^{T_{0}+T^{\varepsilon}+1}  \chi_{[\ell-T^\varepsilon, \ell+T^{\varepsilon}] \cap S_{H,T_0}}(t)\mathrm{d}t \\
		& \ll T^{\varepsilon} H \int_{|u| \leq  T^\varepsilon}  \sum_{\ell \in D_{\alpha}(\delta')} \chi_{S_{H,T_0}}(\ell + u) \mathrm{d}u\\
		& = T^{\varepsilon} H \int_{|u| \leq  T^\varepsilon} |(D_\alpha(\delta') +u) \cap  S_{H,T_0}| \mathrm{d}u.
\end{align*}
We emphasize here that $|(D_\alpha(\delta') +u) \cap  S_{H,T_0}|$ denotes the cardinality of the set (depending on the variable $u$). Again $H$ a priori depends on $\alpha$, but as there are only $O(\log T_0)$ valid choices for $H$, one may select as before a subset of $\mathcal{I}_0$, which we keep denoting as $\mathcal{I}_0$, for which \eqref{case1fail} remains true. Therefore, we may assume without loss of generality that $H$ is independent of $\alpha$.

Consider
\begin{align*}
	W_{\alpha}&:= \int_{|u| \leq  T^\varepsilon} |(D_\alpha(\delta') +u) \cap  S_{H,T_0}| \mathrm{d}u 
	\\
	&= \left(\int_{|u| \leq  T^\varepsilon} |(D_\alpha(\delta') +u) \cap  S_{H,T_0}| \mathrm{d}u\right)^{1-1/q} \left(\int_{S_{H, T_{0}} } \sum_{\ell \in D_{\alpha}(\delta' ) } \chi_{(\ell-T^{\varepsilon}, \ell+T^{\varepsilon})} (t)   \:\mathrm{d} t\right)^{1/q}\\
	& \ll  T^{\varepsilon}  |D_{\alpha}(\delta')|^{1-1/q}  (\mathfrak{m}(S_{H, T_{0}}))^{1/q},
\end{align*}
where $\mathfrak{m}$ stands for the Lebesgue measure and where we have used the trivial estimate $\sum_{\ell \in D_{\alpha}(\delta') } \chi_{(\ell-T^{\varepsilon}, \ell+T^{\varepsilon})} (t) \ll T^{\varepsilon}.$
Hence, via a dyadic argument we can specify $0< \delta'' \ll T^{\varepsilon}$ such that
\begin{equation}\label{del''}
	\delta'' |D_{\alpha}(\delta')|^{1-1/q}  \left(\mathfrak{m}(S_{H, T_{0}})\right)^{1/q} < W_\alpha
	\leq 2\delta'' |D_{\alpha}(\delta')|^{1-1/q}  \left(\mathfrak{m}(S_{H, T_{0}})\right)^{1/q}.
\end{equation}
Again $\delta''$ can be taken independent of $\alpha$ by an appropriate restriction of the index set $\mathcal{I}_0$ and a constraint on the parameter $\delta_{1}$; in fact, we shall require\footnote{We have just derived that $W_\alpha \ll T^{\varepsilon}  |D_{\alpha}(\delta')|^{1-1/q}  \left(\mathfrak{m}(S_{H, T_{0}})\right)^{1/q} \ll T^{c_2}$ for some $c_2 > 0$. Going back through the inequalities we also have the lower bound $W_\alpha \gg  T^{-\varepsilon} H^{-1} \delta_1^{2/q} \gg T^{-c_{1}}$ for some $c_1 > 0$ as $\log H \asymp \log T_0$ and because $\delta_1$ shall later be picked in such a way that $\delta_1 \gg T^{-c}$. Therefore, a dyadic covering for $W_\alpha$ only requires $O(\log T)$ intervals and this enables one to pick a restriction of $\mathcal{I}_0$ such that \eqref{case1fail} remains intact.} from now on that $\delta_1 \gg T^{-c}$ for some $c>0$.

Combining all the above inequalities gives
$$
T^{\varepsilon}\delta''H|D_{\alpha}(\delta')|^{1-1/q}  \left(\mathfrak{m}(S_{H, T_{0}})\right)^{1/q} \gg \frac{\delta_{1}^{2/q}}{\delta'}T_{0}^{B/q} |R_{\alpha}|^{1-1/q},
$$
whence
 \begin{equation*} 
 	| D_{\alpha}(\delta') |  \gg T^{-\varepsilon}\left(\frac{\delta_{1}^{2/q}}{\delta'\delta''}\right)^{\frac{q}{q-1}} |R_{\alpha}|
 \end{equation*}
 follows because the mixed moment estimate $\eqref{mixmo}$ implies $\mathfrak{m}(S_{H,T_0}) \ll H^{-q}T_0^{B}T^{\varepsilon}$.
Thus, using $|D_\alpha(\delta')| \ll |R_\alpha|/\delta'$, we obtain
\begin{align} \label{lowdel'}
	\delta'  \gg T^{-\varepsilon}\frac{\delta_{1}^{2}}{(\delta'')^{q}}.
\end{align}

We can derive another lower bound on $|D_\alpha(\delta')|$. Namely, the trivial bound $W_\alpha \ll T^{\varepsilon}|D_\alpha(\delta')|$ and $W_\alpha \asymp \delta'' |D_{\alpha}(\delta')|^{1-1/q}  \left(\mathfrak{m}(S_{H, T_{0}})\right)^{1/q}$ yield
\begin{equation*}
	| D_{\alpha}(\delta') | \gg T^{-\varepsilon}(\delta'')^{q}\: \mathfrak{m}(S_{H, T_{0}})+ T^{-\varepsilon}\left(\frac{\delta_{1}^{2/q}}{\delta'\delta''}\right)^{\frac{q}{q-1}} |R_{\alpha}|.
\end{equation*}
Together with \eqref{del''} and \eqref{case1fail} this implies 
\begin{align*} 	 \sum_{\alpha \in \mathcal{I}_0} |R_{\alpha}|W_{\alpha}& \gg \sum_{\alpha \in \mathcal{I}_0} |R_\alpha| \delta'' |D_\alpha(\delta')|^{1-1/q} \left(\mathfrak{m}(S_{H,T_0})\right)^{1/q} 
	\\
	 & \gg  T^{-\varepsilon} (\delta'')^{q} \mathfrak{m}(S_{H, T_{0}})\sum_{\alpha \in \mathcal{I}_0} |R_{\alpha}|  
	 + T^{-\varepsilon}\frac{\delta_{1}^{2/q}}{\delta'}  \left(\mathfrak{m}(S_{H, T_{0}})\right)^{1/q} \sum_{\alpha \in \mathcal{I}_0} |R_{\alpha}|^{\frac{2q-1}{q}}  
	 \\ 
	  & \gg  T^{-\varepsilon} (\delta'')^{q}|R| \int_{|u|< T^{\varepsilon}}|(S_{H, T_{0}}- u) \cap \mathbb{Z} |   \:\mathrm{d}u
	 \\
	 &\ \ \ \ \ \ \ \ \ \ \ \ +T^{-\varepsilon}\frac{\delta_{1}^{2/q}\delta_{2}^{1-1/q}}{\delta'}|R|^{\frac{2q-1}{q}}\int_{|u|< T^{\varepsilon}}|(S_{H, T_{0}}- u) \cap \mathbb{Z} | ^{1/q}  \:\mathrm{d}u
\end{align*}
as H\"older's inequality implies $|R|^{(2q-1)/q} \ll \delta_2^{-(q-1)/q} \sum_{\alpha \in \mathcal{I}_0} |R_{\alpha}|^{(2q-1)/q}$ and
$$\int_{|u|< T^{\varepsilon}}|(S_{H, T_{0}}- u) \cap \mathbb{Z} | ^{1/q}  \:\mathrm{d}u \ll \left(\int_{|u|< T^{\varepsilon}}|(S_{H, T_{0}}- u) \cap \mathbb{Z} |\mathrm{d}u\right) ^{1/q} T^{\varepsilon} \ll T^{\varepsilon}\left(\mathfrak{m}(S_{H,T_0})\right)^{1/q}.$$

Recalling the definition of $W_\alpha$, we may therefore find $|u | < T^{\varepsilon}$ such that the set of integers
\begin{equation*} 
	S = (S_{H, T_{0}}- u) \cap \mathbb{Z}
\end{equation*}
satisfies
\begin{align} \label{proDSf}
	 \sum_{\alpha \in \mathcal{I}_0} |R_{\alpha}| | D_{\alpha}(\delta')  \cap S| 
	\gg T^{-\varepsilon} (\delta'')^{q}|R| |S|  
	 +T^{-\varepsilon}\frac{\delta_{1}^{2/q}\delta_{2}^{1-1/q}}{\delta'}|R|^{\frac{2q-1}{q}}|S|^{1/q}.
\end{align}

We keep in mind that we have to multiply by $\delta' \asymp \Delta_\alpha(\ell)/|R_\alpha|$ with $\ell \in D_{\alpha}(\delta')$ to eliminate $\delta'$ and $\delta''$ from the right-hand side by virtue of \eqref{lowdel'}. Now that we have established a lower bound for a multiset of class-I zeros, we shift our attention to an upper bound.

We now select $A$ to be the multiset 
$$ A = \bigcup_{\alpha \in \mathcal{I}_0} \bigcup_{\ell \in S} \bigcup_{t \in \{\Im \rho \mid  \rho \in R_\alpha\}} \left\{\rho = \beta + it' \in R_\alpha , |t' - (t-\ell)| < 1\right\},
$$
where the multiplicity of a zero $\rho$ is according to how many triples $(\alpha,\ell,t)$ produce $\rho$. Therefore, $|A| = \sum_{\alpha \in \mathcal{I}_0} \sum_{\ell \in S} \Delta_{\alpha}(\ell)$. We now apply the machinery from Section \ref{zedete}, in particular \eqref{uppAF}, to find a Dirichlet polynomial $F_{A}(t) = \sum_{N(A) < n \leq 2N(A)} b_n n^{-it}$ with bounded coefficients $b_n$ such that
\begin{align*}
 & N^{2\sigma} \delta' \sum_{\alpha \in \mathcal{I}_0} |R_{\alpha}| |S \cap D_{\alpha}(\delta')|  \ll N(A)^{2\sigma} \sum_{\alpha \in \mathcal{I}_0} \sum_{\ell \in S}  \Delta_{\alpha} (\ell) \\
 &  \ \ \ \ \ \ \ \ \ \ \ll T^{\varepsilon} \sum_{\alpha \in \mathcal{I}_0} \sum_{\ell \in S}  \sum_{t \in \{ \Im \rho : \rho \in R_{\alpha}\} } \sum_{ t'\in \{ \Im \rho : \rho \in R_{\alpha}\} \atop |t'-(t-\ell)| < 1}\left| F_{A}(t')\right|^{2} \\
 & \ \ \ \ \ \ \ \ \ \ \ll T^{\varepsilon} \sum_{\alpha \in \mathcal{I}_0} \sum_{\ell \in S}  \sum_{t \in \{ \Im \rho : \rho \in R_{\alpha}\} } \sum_{ t'\in \{ \Im \rho : \rho \in R_{\alpha}\} \atop |t'-(t-\ell)| < 1}\left(1 + \int_{|v|<\log N(A)} |F_{A}(t' + v)|^{2}\:\mathrm{d}v\right) \\
 & \ \ \ \ \ \ \ \ \ \ \ll T^{\varepsilon}\sum_{\alpha \in \mathcal{I}_0} \sum_{t \in \{ \Im \rho : \rho \in R_{\alpha}\} } \sum_{\ell \in S}  \int_{|v|<T^{\varepsilon}} |F_{A}(t -\ell + v)|^{2} \:\mathrm{d}v.
\end{align*}
In the penultimate transition we applied a Cauchy-Schwarz estimate on Lemma \ref{Bourgpoly} and in the last step we used that there can only be one $t'$ with a given $t$ as the zeros in $R_\alpha$ are well-spaced and that the term with $1$ may be dropped as it only delivers a contribution of at most $T^{\varepsilon} \sum_{\alpha \in \mathcal{I}_0} \sum_{\ell \in S}  \Delta_{\alpha} (\ell)$ which can never be dominant in view of the estimate on the first line.

The remaining sums and integral are estimated via Lemma \ref{BJ} and Lemma \ref{Headou}. This gives
\begin{align*}
 &N^{2\sigma} \delta' \sum_{\alpha \in \mathcal{I}_0} |R_{\alpha}| |S \cap D_{\alpha}(\delta')|  \ll T^\varepsilon \int_{|v|<T^{\varepsilon}} \sum_{\substack{\rho = \beta + it \in R \\ \ell \in S}} |F_{A}(t -\ell + v)|^{2} \:\mathrm{d}v\\
 &  \ \ \ \ \ \ll T^\varepsilon \left(\sum_{\substack{\rho = \beta + it \in R \\ \rho' = \beta' +it' \in R}} \left| \sum_{N(A) < n \leq 2N(A)}  n^{i(t-t')}\right|^{2} \right)^{1/2}  \left(\sum_{\ell, \ell' \in S} \left| \sum_{N(A) < n \leq 2N(A)} n^{i(\ell-\ell')}\right|^{2} \right)^{1/2}\\
 & \ \ \ \ \ \ll T^{\varepsilon} N \left(|R| N + |R|^{2} + |R|^{5/4} T^{1/2}\right)^{1/2}   \left(|S| N + |S|^{2} + |S|^{5/4} T_{0}^{1/2}\right)^{1/2}.
\end{align*}
If we now assume that
\begin{equation*}
	|R| \leq N,
\end{equation*}
and use the condition $N \geq T^{1/2}$, then we find
\begin{align*} 
	N^{2\sigma} \delta' \sum_{\alpha \in \mathcal{I}_0} |R_{\alpha}| |S \cap D_{\alpha}(\delta')| \ll T^{\varepsilon} N^{3/2}|R|^{5/8} \left(|S| N + |S|^{2} + |S|^{5/4} T_{0}^{1/2}\right)^{1/2}.
\end{align*}
Combining this with the lower bound \eqref{proDSf} and eliminating $\delta', \delta''$ through \eqref{lowdel'}, we arrive at
\begin{align*} 
	& N^{2\sigma} \delta_{1}^{2} |R| |S|  
	+ N^{2\sigma}\delta_{1}^{2/q}\delta_{2}^{1-1/q}|R|^{\frac{2q-1}{q}}|S|^{1/q}
	\\
	 & \ \ \ \ \ \  \ll (N^{2} |R|^{5/8}|S|^{1/2} + N^{3/2} |R|^{5/8}|S| + \delta_{2}^{1/4}N^{3/2}T^{1/4} |R|^{5/8}|S|^{5/8})T^\varepsilon .
\end{align*}
One of the three terms on the right is dominant. We now wish to eliminate $|S|$. This can be done as the exponent of $|S|$ for each term on the right-hand side lies between $1/q$ and $1$; note that $q \geq 2$. Therefore, for each term on the right-hand side, $|S|$ can be eliminated by an appropriate interpolation of the two left side terms. After a few calculations one obtains
\begin{align} \label{termcase2}
	|R| \ll T^{\varepsilon} \left(\delta_{1}^{-\frac{8}{7}} \delta_{2}^{-\frac{4}{7}}N^{\frac{16(1-\sigma)}{7}} + \delta_{1}^{-\frac{16}{3}}N^{\frac{4(3-4\sigma)}{3}}+ \delta_{1}^{-\frac{5}{3}}\delta_{2}^{-\frac{1}{6}}N^{\frac{2(3-4\sigma)}{3}}T^{1/3}\right).
\end{align}

\subsection{A large value estimate}\label{LVE}
\

Collecting the contributions \eqref{Dichcon4} and \eqref{termcase2} from Cases 1 and 2, we get the following bound for $|R|$:
\begin{lemma}\label{LVEH} Let $N \geq T^{1/2}$ and $R$ be a set of representative well-spaced class-I zeros. 
If $B_0, B_1 > 0$ and $q_0,q_1 \geq 2$ are parameters for which \eqref{mixmo} holds and  
	$$
		|R| \leq N,
	$$
	 then for any $T^{-c}\ll \delta_{1}<1$ (for some $c>0$) and  $T^{-1} \leq \delta_{2} \leq 1$, 
	we have
	\begin{align*}
		|R| \ll \Big(\delta_{2}^{-1}N^{2-2\sigma} +\delta_{1}^{2}\delta_{2}^{B_{0}-1}T^{B_{0}}N^{(3-4\sigma)q_{0}/2} + \delta_{1}^{2}\delta_{2}^{B_{1}-1}T^{B_{1}}N^{(3-4\sigma)q_{1}/2}  
		\\
		\notag
		 + \ \delta_{1}^{-\frac{8}{7}} \delta_{2}^{-\frac{4}{7}}N^{\frac{16(1-\sigma)}{7}} + \delta_{1}^{-\frac{16}{3}}N^{\frac{4(3-4\sigma)}{3}}+ \delta_{1}^{-\frac{5}{3}}\delta_{2}^{-\frac{1}{6}}N^{\frac{2(3-4\sigma)}{3}}T^{1/3}\Big)T^\varepsilon.
	\end{align*}
\end{lemma} 
\vspace{1px}

\subsection{Proof of Theorem \ref{main1}}\label{pth1}
\

From the analysis of Section \ref{zedete} it only remains to find an estimate for the representative well-spaced class-I zeros and from Section \ref{Ivices} we may suppose that \eqref{IvrelenN} holds and that $\sigma  < (3q^{\ast}-4)/(4q^{\ast}-4B^{\ast}-4)$. We also recall the restriction $\sigma \geq  1 - q_j/(4B_j + 4q_j -4)$, $j = 0,1$, we encountered in Section \ref{Ivices}. We consider first the case that $|R| \leq N$.
 Applying Lemma \ref{LVEH} with admissible parameters $q_0, q_1, B_0, B_1$ gives
\begin{align*}
	|R| \ll (\delta_{2}^{-1}N^{2-2\sigma} &+\delta_{1}^{2}\delta_2^{B_0 - 1}T^{B_0}N^{(3-4\sigma)q_0/2} + \delta_{1}^{2}\delta_{2}^{B_1}T^{B_1 - 1}N^{(3-4\sigma)q_1/2} 
	\\
	& + \delta_{1}^{-\frac{8}{7}} \delta_{2}^{-\frac{4}{7}}N^{\frac{16(1-\sigma)}{7}} + \delta_{1}^{-\frac{16}{3}}N^{\frac{4(3-4\sigma)}{3}}+ \delta_{1}^{-\frac{5}{3}}\delta_{2}^{-\frac{1}{6}}N^{\frac{2(3-4\sigma)}{3}}T^{1/3})T^{\varepsilon}.
\end{align*}
Let us choose for $T^{1/2} < N \leq T$ the parameters\footnote{If $T < N < T^{1+\varepsilon}$, we select the same parameters as were $N = T$, that is $\delta_1 = T^{((4q-4)\sigma-(3q+2B-4))/4}$ and $\delta_2 = 1$. The verification of the density hypothesis then becomes the same calculation as for $N = T$ except for some extra factors that can be absorbed in $T^{\varepsilon}$.} $\delta_{1}$ and $\delta_{2}$ in such a way that
$$
 \delta_{2}^{-1}N^{2-2\sigma} = \delta_{1}^{2}\delta_{2}^{B-1}T^{B}N^{(3-4\sigma)q/2} = T^{2(1-\sigma)}, 
$$
where $(q,B)$ is the couple $(q_j, B_j)$, $j = 0,1$, for which $\delta_{2}^{B-1} T^{B}N^{(3-4\sigma)q/2}$ is maximal. This is equivalent to 
\begin{align*}
	\delta_{2}= N^{2-2\sigma}T^{2\sigma-2}, \ \ \ \delta_{1}
=N^{\frac{(4B-4 +4q)\sigma-(4B-4+3q)}{4}}T^{\frac{B-2B\sigma}{2}},
\end{align*} 
whence $T^{-c}<\delta_{1} < 1$ (with e.g. $c=\max_{j=0,1}\{7q_j/4+3B_j/2\} +2$) and $T^{-1} \leq \delta_{2} \leq 1$ in view\footnote{One may verify through a monotonicity argument that also $\delta_1 \leq 1$ even if 
$(q,B)\neq (q^{\ast},B^{\ast})$.} of \eqref{IvrelenN}. Inserting this choice in the estimate for $|R|$ gives, as $N \geq T^{1/2}$, and further imposing the restriction $\sigma \geq 1 - q/(4B+4q)$, 
\begin{align*}
	|R| & \ll  \Big(T^{2(1-\sigma)} + N^{\frac{(8B +6q)-(8B+8q)\sigma}{7}}T^{\frac{(8B-8)\sigma-(4B-8)}{7}} + N^{\frac{(16B-4+12q)-(16B+16q)\sigma }{3}}T^{\frac{16B\sigma-8B}{3}}\\
	& \ \ \ \ \ \ + N^{\frac{(20B+15q)-(20B+20q+8)\sigma}{12}}T^{\frac{(10B-2)\sigma-(5B-4)}{6}}\Big)T^{\varepsilon}
	\notag
	\\
	& \ll \left(T^{2(1-\sigma)} + T^{\frac{(3q+8)-(4q-4B+8)\sigma}{7}} + T^{\frac{(6q-2)-(8q-8B)\sigma}{3}}+ T^{\frac{(15q+16)-(20q-20B+16)\sigma}{24}}\right)T^{\varepsilon}
	\notag
	\\
	 & \ll T^{2(1-\sigma)+\varepsilon},
\end{align*}
where the second, third and fourth summand give new restrictions on $\sigma$. Summarizing, for $j = 0,1$, we obtain the set of constraints
$$\sigma \geq \frac{3q_j-6}{4q_j-4B_j-6}, \quad \sigma \geq \frac{3q_j-4}{4q_j-4B_j-3}, \quad \sigma \geq \frac{15q_j-32}{20q_j-20B_j-32}, \quad \sigma \geq \frac{3q_j+4B_j}{4q_j+4B_j},$$
provided that also $q_j>B_j+8/5$ which ensures that the denominators in the above fractions are all positive, as otherwise we would not obtain any range for $\sigma$. 
The density hypothesis holds under these restrictions\footnote{The restriction $\sigma \geq (3q_j + 4B_j)/(4q_j+4B_j)$ is implied by $\sigma \geq (3q_j - 6)/(4q_j -4B_j -6)$ if, say, $B_j\geq 3/4$.} for $\sigma$ and if $|R| \leq N$.

Now, suppose that $|R| > N$. As $N \geq T^{1/2}$, this implies that $|R| > T^{1/2}$. Select now a subset of representative well-spaced class-I zeros $R'$ such that $|R'| = \lfloor T^{1/2}\rfloor$. Now $|R'| \leq N$ and the entire analysis above can be performed for $R'$ to give $|R'| \ll T^{2(1-\sigma) + \varepsilon} \ll T^{1/2 - \varepsilon}$, if $\sigma > 3/4$ say, which is impossible (for large enough $T$). Therefore $|R|$ must have been smaller than $N$ to begin with.

It only remains to pick the best possible $q_0,q_1, B_0$ and $B_1$. In order to find admissible values, we are going to appeal to \cite[Thm.  8.2, p. 206]{Ivicbook}. Given an \emph{exponent pair} $(a,b)$, this result guarantees that $q_0 = 6$, $B_0 = 1 + \varepsilon$, $q_1 = 2(1+ 2a+2b)/a$ and $B_1 = (a+b)/a + \varepsilon$ are admissible values for \eqref{mixmo}. The couple $(q_0,B_0) = (6, 1 +\varepsilon)$ subsequently gives the restriction $\sigma \geq 6/7$.

For the other couple $(q_1,B_1)$ it turns out that the restriction $\sigma \geq (3q_{1}-6)/(4q_{1}-4B_{1}-6)$ is the critical one. Rewriting this range in terms of the exponential pair $(a,b)$ gives $\sigma \geq 1 - 1/(3a + 6b +4)$. Our task is therefore to minimize $a+2b$. To the best of our knowledge, the exponent pair $(55/194 + \varepsilon, 110/194 + \varepsilon)$ is the best available choice\footnote{The preprint \cite{TrudgianYang} claims that $(1/4 + \varepsilon ,7/12 + \varepsilon)$ is also an exponent pair, which would deliver a lower value for $a+2b$. This exponent pair is derived from applying process B on $(1/12 + \varepsilon, 3/4 + \varepsilon)$ and this was, according to the preprint, supposed to 
have been shown in \cite{Robert}. However, it is unclear how the exponent pair $(1/12 + \varepsilon, 3/4 + \varepsilon)$ follows from \cite[Thm. 1]{Robert}. Robert also does not claim his result implies that $(1/12 + \varepsilon, 3/4 + \varepsilon)$ would be an exponent pair.} at the moment. This exponent pair is derived from first applying Process A and then Process B on the exponent pair $(13/84 + \varepsilon, 55/84 + \varepsilon)$ that Bourgain established in \cite{Bourgain2016}. The parameters then become $q_1 = 1048/55$, $B_1 = 3 + \varepsilon$ and $H(T) = T^{55/359}$. One ultimately finds that the density hypothesis is valid in the range $\sigma \geq 1407/1601$. This concludes the proof of Theorem \ref{main1}.

\appendix
\section{A zero-density estimate for the Riemann zeta function}\label{eszeta}

In this appendix we establish the zero-density estimate $N(\sigma, T) \ll T^{\frac{24(1-\sigma)}{30 \sigma -11}+\varepsilon}$ for the Riemann zeta function for a broader range for $\sigma$ than what Ivi\'c initially obtained in \cite[Thm.  11.2, Eq. 11.31]{Ivicbook}. The precise exponent in $T$ comes from an optimization with respect to the specific technology that Ivi\'c employed in his proof; the crucial factors are the exponent pair $(2/7,4/7)$ from which \cite[Thm.  8.2]{Ivicbook} delivers a bound that was used for the estimation of the class-II zeros, and the specific moments $q_0 = 6$ and $q_1 = 19$ chosen for the mixed moment argument \eqref{mixmo}. For specific $\sigma$ in the interval under consideration here, it should be possible to optimize the exponent pair and the mixed moment exponents to obtain a better exponent for $T$ in the final zero-density estimate. This is however not the main focus of the appendix and we decided not to pursue this here. Obtaining the best zero-density estimates for the Riemann zeta function by selecting the optimal exponent pairs is one of the objectives of the preprint \cite{TrudgianYang}. Our main goal here is to illustrate how Bourgain's method allows one to use Heath-Brown's double zeta sum estimate Lemma \ref{Headou} to increase the range of validity of the zero-density estimate
$$ N(\sigma, T) \ll T^{\frac{24(1-\sigma)}{30 \sigma -11}+\varepsilon},
$$
from $0.8908 \approx 155/174 \leq \sigma \leq 17/18$ to $0.8885 \approx 279/314 \leq \sigma \leq 17/18$. As this is the only place where we modify Ivi\'c's argument, we do not achieve a lower exponent for $T$ in the zero-density estimate.

\subsection{Some modifications}\label{modi}

\

As the proof method is very similar to the proof of the density hypothesis in the range $\sigma \geq 1407/1601$ for $L(s,f)$ discussed in detail in the paper, we only point out the differences. Moreover, since Ivi\'c had already shown Theorem \ref{main2} when $155/174 \leq \sigma \leq 17/18$, we shall only work under the hypothesis $279/314 \leq \sigma < 155/174$.  

We employ the same zero-detection method as in Section \ref{zedete}, with the obvious changes that $L(s,f)$ is replaced by $\zeta(s)$ and $\mu_f$ by the classical M\"obius function. Now in the calculation of $\zeta(s) M_X(s)$, when shifting the line of integration, we encounter an additional pole at $z=1-s$ which delivers the extra term $M_{X}(1)Y^{1-s}\Gamma(1-s)$. When $|\Im s| \geq \log^{2}T$, this term is however still $o(1)$ as $T \rightarrow \infty$. So, if a $\zeta$-zero $\rho = \beta + it$ satisfies $|t| \geq \log^{2}T$, it must still be either a class-I zero or a class-II zero. Instead of \eqref{NOze}, we thus obtain
\begin{equation*} 
	N(\sigma, T) \ll (|R_1|+|R_2|+1)T^{\varepsilon},
\end{equation*}	
as  $N(1/2, \log^{2}T)  \ll \log^{3} T$, say.

When Ivi\'c handles the class-II zeros $R_{2}$, he takes $Y=T^{\frac{6}{30\sigma -11}}$ and $\nu=2$. Therefore \eqref{LengthPolyF} is replaced by 
\begin{equation} \label{zetaLengthPolyF}
	T^{\frac{8}{30\sigma -11}} \leq N \leq T^{\frac{12}{30\sigma -11}+\varepsilon}.
\end{equation}
As explained above, the estimation of the class-II zeros is slightly different than in the paper, but following the argument of Ivi\'{c} \cite[Section 11.2]{Ivicbook}, one obtains that $|R_{2}|$ is bounded by \cite[Eq. 11.41]{Ivicbook},
 \begin{equation*}
 	|R_{2}| \ll (TY^{3-6\sigma}+T^{3}Y^{19(1/2-\sigma)})T^{\varepsilon} \ll T^{24(1-\sigma)/(30\sigma-11)+\varepsilon}.
 \end{equation*}
With the same technology as in Section \ref{Ivices} with the mixed moment parameters $q_0 = 6, A_0 = 1+ \varepsilon, q_1 = 19, A_1 = 3 + \varepsilon$, one finds an estimate for the representative well-spaced class-I zeros \cite[Eq. 11.42]{Ivicbook}, 
\begin{equation*}
	|R| \ll (N^{2-2\sigma}+TN^{(65-84\sigma)/6})T^{\varepsilon},
\end{equation*}
which also gives the desired $T^{24(1-\sigma)/(30\sigma -11)+\varepsilon}$ estimate, provided $N$ satisfies 
\begin{equation*}
	N \geq T^{\frac{6}{65-84\sigma}\cdot \frac{35-54\sigma}{30\sigma-11}}.
\end{equation*}
In view of \eqref{zetaLengthPolyF} the above estimate is always valid if $\sigma \geq 155/174$ and this concludes Ivi\'c's argument. For the remaining range, we may thus assume
\begin{equation} \label{zetaIvrelenN}
	T^{\frac{8}{30\sigma -11}} \leq N \leq T^{\frac{6}{65-84\sigma}\cdot \frac{35-54\sigma}{30\sigma-11}}.
\end{equation}

The analysis of Section \ref{BouD} is mostly analogous, except at the end in the treatment of case 2 where instead of the bound $N \geq T^{1/2}$ we shall use $N \geq T^{\frac{8}{30\sigma -11}}$ and instead of $|R| \leq N$ we use the modified
\begin{align} \label{zetaassum RN}
	T^{24(1-\sigma)/(30\sigma -11)-\varepsilon} \leq |R| \leq N.
\end{align}
This results in the bound 
\begin{align*} 
	& N^{2\sigma} \delta' \sum_{\alpha \in \mathcal{I}_0} |R_{\alpha}| |S \cap D_{\alpha}(\delta')| \\
	& \ \ \ \ \ \ \ll T^{\varepsilon} N^{3/2}|R|^{1/2}(1+R^{1/4}N^{-1}T^{1/2})^{1/2} \left(|S| N + |S|^{2} + |S|^{5/4} T_{0}^{1/2}\right)^{1/2}
	\\
	& \ \ \ \ \ \ \ll  T^{\varepsilon} N^{3/2}|R|^{1/2}(1+R^{1/4}T^{-8/(30\sigma -11)}T^{1/2})^{1/2} \left(|S| N + |S|^{2} + |S|^{5/4} T_{0}^{1/2}\right)^{1/2}
	\\
	& \ \ \ \ \ \ \ll  T^{\varepsilon} N^{3/2}|R|^{5/8}T^{\frac{30\sigma -27}{4(30\sigma -11)}}\left(|S| N + |S|^{2} + |S|^{5/4} T_{0}^{1/2}\right)^{1/2}.
\end{align*}
The lower inequality for $|R|$ in \eqref{zetaassum RN} was only imposed to guarantee $1 \ll R^{\frac{1}{4}}T^{\frac{-8}{30\sigma -11}}T^{\frac{1}{2}}$. Combining this with the lower inequality 
\eqref{proDSf}, we find 

\begin{align*} 
	 & N^{2\sigma} \delta_{1}^{2} |R| |S|  
	+ N^{2\sigma}\delta_{1}^{2/q}\delta_{2}^{1-1/q}|R|^{\frac{2q-1}{q}}|S|^{1/q}\ll \Big( N^{2} |R|^{5/8}|S|^{1/2}T^{\frac{30\sigma -27}{4(30\sigma -11)}} \\	
	& \ \ \ \ \ \ \ \ \ \ \ \ \ \ \ \ \ \ \ \ \ \ \ \  + N^{3/2} |R|^{5/8}|S|T^{\frac{30\sigma -27}{4(30\sigma -11)}} 
	+ \delta_{2}^{1/4}N^{3/2}T^{1/4} |R|^{5/8}|S|^{5/8}T^{\frac{30\sigma -27}{4(30\sigma -11)}}\Big) T^{\varepsilon}.
\end{align*}
With a suitable interpolation to eliminate $|S|$, one then finds after a few calculations
\begin{align*} 
	|R| \ll  T^{\varepsilon} \left(\delta_{1}^{-\frac{8}{7}} \delta_{2}^{-\frac{4}{7}}N^{\frac{16(1-\sigma)}{7}}T^{\frac{6(10\sigma -9)}{7(30\sigma -11)}} + \delta_{1}^{-\frac{16}{3}}N^{\frac{4(3-4\sigma)}{3}}T^{\frac{2(10\sigma -9)}{30\sigma -11}} + \delta_{1}^{-\frac{5}{3}}\delta_{2}^{-\frac{1}{6}}N^{\frac{2(3-4\sigma)}{3}}T^{\frac{1}{3}+\frac{10\sigma -9}{30\sigma -11}}\right).
\end{align*}

\subsection{The large value estimate}\label{zetaLVE}

\

The corresponding large value estimate then becomes
\begin{lemma} \label{zetaLVEH} Let $N \geq T^{\frac{8}{30\sigma -11}}$ and $R$ be a set of representative well-spaced class-I zeros. 
If $(q_{0}, A_{0})$ and $(q_{1}, A_{1})$ satisfy \eqref{mixmo} and
	$$
	T^{24(1-\sigma)/(30\sigma -11)-\varepsilon} \leq |R| \leq N,
	$$
	then for any $T^{-c}\ll \delta_{1}<1$ (for some $c>0$) and  $T^{-1} \leq \delta_{2} \leq 1$,
		\begin{align*} 
		|R| & \ll \big( \delta_{2}^{-1}N^{2-2\sigma} +\delta_{1}^{2}\delta_{2}^{A_{0}-1}T^{A_{0}}N^{(3-4\sigma)q_{0}/2} + \delta_{1}^{2}\delta_{2}^{A_{1}-1}T^{A_{1}}N^{(3-4\sigma)q_{1}/2} 
		\notag
		\\
		& \ \ \  + \delta_{1}^{-\frac{8}{7}} \delta_{2}^{-\frac{4}{7}}N^{\frac{16(1-\sigma)}{7}}T^{\frac{6(10\sigma -9)}{7(30\sigma -11)}}+ \delta_{1}^{-\frac{16}{3}}N^{\frac{4(3-4\sigma)}{3}}T^{\frac{2(10\sigma -9)}{30\sigma -11}} + \delta_{1}^{-\frac{5}{3}}\delta_{2}^{-\frac{1}{6}}N^{\frac{2(3-4\sigma)}{3}}T^{\frac{1}{3}+\frac{10\sigma -9}{30\sigma -11}}\big) T^{\varepsilon}.
	\end{align*}
\end{lemma} 

\subsection{Proof of Theorem \ref{main2}}\label{pth2}

\

	As we already have a bound for the class-II zeros, we are only required to estimate $|R|$. Suppose first that $T^{24(1-\sigma)/(30\sigma -11)-\varepsilon} \leq  |R| \leq N$. We apply the large value estimate Lemma \ref{zetaLVEH} with $(q_{0}, A_{0})=(6, 1+\varepsilon)$ and $(q_{1}, A_{1})=(19, 3+\varepsilon)$. 
	This gives
	\begin{align*} 
		 \quad |R|  & \ll (\delta_{2}^{-1}N^{2-2\sigma} +\delta_{1}^{2}TN^{9-12\sigma} + \delta_{1}^{2}\delta_{2}^{2}T^{3}N^{\frac{19(3-4\sigma)}{2}} + \delta_{1}^{-\frac{8}{7}} \delta_{2}^{-\frac{4}{7}}N^{\frac{16(1-\sigma)}{7}}T^{\frac{6(10\sigma -9)}{7(30\sigma -11)}}
		\notag
		\\
		& \ \ \ \ \ \  + \delta_{1}^{-\frac{16}{3}}N^{\frac{4(3-4\sigma)}{3}}T^{\frac{2(10\sigma -9)}{30\sigma -11}} + \delta_{1}^{-\frac{5}{3}}\delta_{2}^{-\frac{1}{6}}N^{\frac{2(3-4\sigma)}{3}}T^{\frac{1}{3}+\frac{10\sigma -9}{30\sigma -11}})T^{\varepsilon}.
	\end{align*}
	
	We choose the parameters $\delta_{1}$ and $\delta_{2}$ in such a way that
	$$
	\delta_{2}^{-1}N^{2-2\sigma} = \delta_{1}^{2}\delta_{2}^{2}T^{3}N^{\frac{19(3-4\sigma)}{2}} = T^{\frac{24(1-\sigma)}{30\sigma -11}} .
	$$
	This is equivalent to 
	\begin{align*} 
		\delta_{2}= N^{2-2\sigma}T^{\frac{24(\sigma-1)}{30\sigma -11}}, \ \ \ \delta_{1} =\delta_{2}^{-1}N^{\frac{19(4\sigma-3)}{4}}T^{\frac{57(1-2\sigma)}{2(30\sigma -11)}}=N^{\frac{84\sigma-65}{4}}T^{\frac{105-162\sigma}{2(30\sigma -11)}},
	\end{align*} 
	whence $T^{-c} <\delta_{1} < 1$, for $c= 13$, say, and $T^{-1} \leq \delta_{2} \leq 1$ in view of \eqref{zetaIvrelenN}. We find, using $N \geq T^{\frac{8}{30\sigma - 11}}$,
	\begin{align*}
		|R|   &\ll (T^{\frac{24(1-\sigma)}{30\sigma -11}}+TN^{9-12\sigma} +  N^{\frac{138-176\sigma}{7}}T^{\frac{612\sigma-378}{7(30\sigma -11)}} + N^{\frac{272-352\sigma }{3}}T^{\frac{452\sigma-298}{30\sigma-11}}
		\notag
		\\
		 & \ \ \ + N^{\frac{345-448\sigma }{12}}T^{\frac{282\sigma-185}{2(30\sigma-11)}+\frac{1}{3}})T^{\varepsilon}
		\notag
		\\
		& \ll  (T^{\frac{24(1-\sigma)}{30\sigma -11}}+TT^{\frac{8(9-12\sigma)}{30\sigma-11}} + T^{\frac{8(138-176\sigma)}{7(30\sigma-11)}}T^{\frac{612\sigma-378}{7(30\sigma -11)}}+ T^{\frac{8(272-352\sigma) }{3(30\sigma-11)}}T^{\frac{452\sigma-298}{30\sigma-11}}
		\notag
		\\
		&\ \ \ + T^{\frac{2(345-448\sigma) }{3(30\sigma-11)}}T^{\frac{282\sigma-185}{2(30\sigma-11)}+\frac{1}{3}})T^{\varepsilon}
		\notag
		\\
		 & \ll T^{\frac{24(1-\sigma)}{30\sigma-11}+\varepsilon}.
	\end{align*}
The second, third, fourth and fifth summand give respectively the conditions $\sigma \geq 37/42 \approx 0.8809$, $\sigma \geq 279/314 \approx 0.8885$, $\sigma \geq 605/694 \approx 0.8717$ and $\sigma \geq 659/742 \approx 0.8881$ on the range of validity of this estimate. Therefore the desired zero-density is valid under the condition $T^{24(1-\sigma)/(30\sigma -11)-\varepsilon}\leq  |R| \leq N$.

If $|R| \leq T^{24(1-\sigma)/(30\sigma -11)-\varepsilon}$, there is nothing left to prove and if $|R| \geq N$ one may take as before a sufficiently large subset of representative well-spaced class-I zeros to obtain a contradiction as $N \geq T^{\frac{8}{30\sigma - 11}} \gg T^{\frac{24(1-\sigma)}{30 \sigma -11} + \varepsilon}$ in the range under question for $\sigma$.
This completes the proof of Theorem \ref{main2}.

\end{document}